\documentclass[12pt]{article}
\usepackage{amsmath,amssymb,amsfonts,amsthm}

\makeatletter

\newtheorem{theorem}{Theorem}
\newtheorem{lemma}[theorem]{Lemma}

\newtheorem{corollary}[theorem]{Corollary}
\theoremstyle{definition}
\newtheorem{definition}[theorem]{Definition}
\newtheorem{example}[theorem]{Example}
\newtheorem{remark}[theorem]{Remark}

\newcommand{\m}{M}

\newcommand{\R}{{\mathbb R}}

\begin{document}

\markboth{A.M. Candela, A. Romero and M. S\'anchez}
{Completeness of Relativistic Particles}

\title{Completeness of Trajectories of Relativistic \\ Particles
Under Stationary Magnetic Fields\footnote{Partially supported by Spanish Grants with FEDER
funds MTM2010-18099 (MIC\-INN) and P09-FQM-4496 (J. de Andaluc\'\i a).}}

\author{{\bf A.M. Candela\footnote{Partially supported by M.I.U.R. (research funds ex 40\% and 60\%) and by
the G.N.A.M.P.A. Research Project ``Analisi Geometrica sulle Variet\`a
di Lorentz ed Applicazioni alla Relativit\`a Generale 2012''.}, A. Romero$^{,1}$  and
M. S\'anchez$^{2}$}\\
\\
{\small $^\dagger$Dipartimento di Matematica, Universit\`a degli Studi di Bari ``A. Moro'',}\\
{\small Via E. Orabona 4, 70125 Bari, Italy}\\
{\small \texttt{candela@dm.uniba.it}}\\
{\small $^{1,2}$Departamento de Geometr\'{\i}a y Topolog\'{\i}a,}\\
{\small Facultad de Ciencias, Universidad de Granada,}\\
{\small 18071 Granada, Spain}\\
{\small $^1$\texttt{aromero@ugr.es},
$^2$\texttt{sanchezm@ugr.es}}}
\date{}

\maketitle

\begin{abstract}
The second order differential equation $\frac{D\dot\gamma}{dt}(t)\
=\ F_{\gamma(t)}(\dot\gamma(t)) - \nabla V(\gamma(t))$ on a
Lorentzian manifold describes, in particular, the dynamics of
particles under the action of a electromagnetic field $F$ and a
conservative force $-\nabla V$. We provide a  first study on the
extendability of its solutions, by imposing  some natural
assumptions.
\end{abstract}
\bigskip

\noindent
{\it Key words and phrases:}{Second order differential equation; Lorentzian manifold;
completeness of inextensible trajectories; electromagnetic field;
stationary and conformal vector fields.}

\section{Introduction}

Let $(\m,g)$ be a (connected, finite--dimensional) Lorentzian
manifold and denote by $\pi : \m \times \R \longrightarrow \m$ the
natural projection. Giving a (1,1) smooth tensor field $F$ along
$\pi$ and a smooth vector field $X$ along $\pi$, let us consider
the second order differential equation
\[
\hspace*{2.9cm}\frac{D\dot\gamma}{dt}(t)\ =\ F_{(\gamma(t),t)}(\dot\gamma(t)) +
X_{(\gamma(t),t)},\hspace*{2.9cm}\mathrm{(E)}
\]
where $D/dt$ denotes the covariant derivative along $\gamma$
induced by the Levi--Civita connection of $g$ and $\dot\gamma$
represents the velocity field along $\gamma$.

Taking $p\in \m$ and $v\in T_p\m$, there exists a unique
inextensible smooth curve $\gamma : I \to \m$, $0\in I$, solution
of (E) which satisfies the initial conditions
\[
\gamma(0) = p,\quad \dot\gamma(0) = v.
\]
Such a curve is called {\sl complete} if $I = \R$ and
{\sl forward} (resp. {\sl backward}) {\sl complete} when $I = (a,b)$ with
$b=+\infty$ (resp. with $a=-\infty$).

In the recent article \cite{CRSfeb2012}, we have investigated the
completeness of the inextensible solutions of $(E)$ when $(\m,g)$
is a Riemannian manifold both, in the autonomous and in the
non--autonomous case, in particular when $X$ derives from a
potential. Furthermore, such results have been applied for
studying a special class of Lorentzian manifolds, which generalize
the parallely propagated waves (briefly, {\em pp--waves}), whose geodesic
completeness follows from the completeness of the trajectories of
a suitable version of $(E)$ stated in a Riemannian manifold (see
\cite{CRSfeb2012, CRSago2012}).

Here, our aim is investigating directly the completeness of the
inextensible solutions of $(E)$ in a Lorentzian manifold.
Nevertheless, such a problem is much more complicated in this
case. In fact, it includes, for example, the geodesic completeness
of $(\m,g)$ (i.e., the case $F = 0$ and $X = 0$). This problem is
handled in the Riemannian case by means of the classical
Hopf--Rinow theorem, but nothing similar holds in the Lorentzian one
(see the survey \cite{CS}). 
So, in order to consider a manageable Lorentzian
 problem, some additional assumptions will be made. This will
allow to introduce a new type of results, which may be extended in
further works.

So, as a physically meaningful framework, we will assume that $X$
derives from an (autonomous) potential ($X = - \nabla V$) and $F$
is skew--adjoint. In particular, when   timelike trajectories are
taken into account, $(E)$ will describe the dynamics of
relativistic particles subject to an electromagnetic field $F$
(i.e., being accelerated through a term which corresponds to the
Lorentz force law) plus  an exterior potential $V$
 (see, e.g., \cite[p. 88]{SW}).

This framework still includes the problem of geodesic
completeness. Hen\-ce, in order to prove the completeness of the
solutions of $(E)$, we will select a representative case were the
problem  has been solved for geodesics, namely, the case when a
timelike conformal vector field $K$ exists, the so--called {\em conformastationary 
spacetimes} (see \cite{RS}). In this ambient, the hypothesis
$F(K)=0$ means that {\em the  conformastationary observers} (those
moving along the integral curves of $K$) {\em do not feel any
electric force, but only magnetic ones}.

Finally, as a simplifying hypothesis, we will assume that $\m$ is
compact. About this hypothesis, recall that, on one hand, the
techniques will be extensible to the non--compact case (in the
spirit of \cite{RS}) and, on the other hand, the compact case is not by
any means trivial, even for geodesic completeness (see \cite{M,C,Kl}).

\begin{theorem}\label{main}
Let $(\m,g)$ be a compact Lorentzian manifold, $F$ a smooth
$(1,1)$ skew--adjoint tensor field on $\m$ and $V : \m \to \R$ a
smooth potential. If a timelike conformal Killing vector field $K$
exists such that $F(K) = 0$, then each inextensible solution of
\[
\hspace*{2.9cm}\frac{D\dot\gamma}{dt}(t)\ =\ F_{\gamma(t)}(\dot\gamma(t))
- \nabla V(\gamma(t))\hspace*{2.9cm}\mathrm{(E_0)}
\]
must be complete.
\end{theorem}

This paper is organized in the following way:  the main results in
the Riemannian case  are recalled in Section \ref{sec2} , the
specific Lorentzian tools are introduced in Section~\ref{sec3},
and the proof of Theorem \ref{main} is developed in Section
\ref{sec4}.

\section{Background about the Riemannian Case}\label{sec2}

In this section we outline the main results  obtained in the
Riemannian case (for their proofs, see \cite{CRSfeb2012}), even
though only some of the tools will be applicable to the Lorentzian
one. To this aim, we need some definitions.

Firstly, we recall that the $(1,1)$ tensor field $F$ can be decomposed
as
\[
F\ =\ S+H,
\]
where $S$ is the self--adjoint part of $F$ with respect to
$g$, and $H$ is the skew--adjoint one.

From now till the end of this section, let $(\m,g)$ be a Riemannian manifold.
Taking any $t \in \R$ and considering the slice $\m\times \{t\}$, we denote
\[\begin{split}
&S_{\mathrm{sup}}(t) := \sup_{\underset{\|v\|=1}{p\in\m, v\in T_p\m}} g(v,S_{(p,t)} v),
\quad S_{\mathrm{inf}}(t) := \inf_{\underset{\|v\|=1}{p\in\m, v\in T_p\m}} g(v,S_{(p,t)} v), \\
& \| S(t)\| : =
\max\{|S_{\mathrm{sup}}(t)|,|S_{\mathrm{inf}}(t)|\}.
\end{split}
\]
We say that $S$ is {\em bounded} (resp. {\em upper bounded}; {\em lower bounded}) {\em
along finite times} when, for each $T>0$ there exists a constant
$N_T$ such that
\begin{equation}\label{bx}
\|S(t)\|  < N_T \ \; (\hbox{resp.} \;
S_{\mathrm{sup}}(t)<N_T ;\ \; S_{\mathrm{inf}}(t)>- N_T)
\ \; \hbox{for all $t\in [-T,T]$.}
\end{equation}

Moreover, if $X$ is a vector field along $\pi$ and $d$ denotes the
distance canonically associated to the Riemannian metric $g$, we
say that $X$ {\em grows at most linearly in $\m$ along finite
times} if for each $T>0$ there exists $p_0\in \m$ and some
constants $A_T, C_T>0$ such that
\begin{equation}\label{bx2}
\sqrt{g(X_{(p,t)},X_{(p,t)})}\ \leq\ A_T \ d(p,p_0) +  C_T \quad \text{for all} \quad
(p,t)\in \m\times [-T,T].
\end{equation}
Obviously, conditions \eqref{bx}, \eqref{bx2} are independent of
the chosen point $p_0$.

\begin{theorem}\label{A1}
Let $(\m,g)$ be a complete Riemannian manifold and consider a
$(1,1)$ tensor field $F$ and a vector field $X$ both
time--dependent and smooth.

If $X$ grows at most linearly in $\m$ along finite times and the
self--adjoint part $S$ of $F$ is bounded (resp. upper bounded;
lower bounded) along finite times, then each inextensible solution
of {\rm (E)} must be complete (resp. forward complete; backward
complete).

In particular, if $\m$ is compact then any inextensible solution of
{\rm(E)} is complete for any $X$ and $F$.
\end{theorem}

Let us point out that the hypotheses in Theorem \ref{A1} are optimal
(see \cite[Example 1]{CRSfeb2012}) and
do not require that $X$ is conservative, i.e. it depends on a potential.
Now, let $V:\m\times\R \rightarrow \R$ be a smooth time--dependent
potential, and emphasize as $\nabla^{\m}V$ the gradient of the
function $p \in \m \mapsto V(p,t)\in \R$, for each fixed $t\in \R$.
In this setting, Theorem \ref{A1} reduces to the following
result.

\begin{corollary}\label{A10}
Let $(\m,g)$ be a complete Riemannian manifold, consider a $(1,1)$
tensor field $F$, eventually time--dependent, with self--adjoint
component $S$, and let $V:\m\times\R \rightarrow \R$ be a smooth
potential. If $S$ is bounded along finite times and $\nabla^\m
V(p,t)$ grows at most linearly in $\m$ along finite times, then
each inextensible solution of $\mathrm{(E)}$ must be complete.
\end{corollary}

The proof of Theorem \ref{A1} is based on the similar result proved
when both $F$ and $X$ are time--independent as the non--autonomous case $(E)$ can be
reduced to the autonomous one by working on the manifold $\m\times \R$
(see \cite[Proposition 1]{CRSfeb2012}).
On the contrary, when $X$ is a time--dependent conservative vector field,
another result can be stated but with a direct proof
in the non--autonomous case.
In order to describe such a result,
we need a further notion.

A function $U: \m\times \R\rightarrow \R$ {\em grows at most
quadratically along finite times} if for each $T>0$ there exist $p_0\in\m$
and some constants $A_T, C_T>0$ such that
\[
U(p,t)\ \leq\ A_T \ d^2(p,p_0) +  C_T \quad \text{for all} \;
(p,t)\in \m\times [-T,T]
\]
(again, this property is independent of the chosen $p_0$).

\begin{theorem}\label{A02}
Let $(\m,g)$ be a complete Riemannian manifold, $F$ a smooth
time--dependent $(1,1)$ tensor field on $\m$ with self--adjoint
component $S$ and $V:\m\times\R \to\R$ a smooth time--dependent
potential.

Assume that $S$ is bounded (resp. upper bounded; lower bounded)
along finite times and $-V$ grows at most  quadratically along
finite times.

If also $|\partial V/\partial t|  :\m\times \R\rightarrow \R$ (resp.
$\partial V/\partial t$; -$\partial V/\partial t$) grows at most
quadratically along finite times, then each inextensible solution
of
\[
\hspace*{3cm}\frac{D\dot\gamma}{dt}(t)\ =\ F_{(\gamma(t),t)}(\dot\gamma(t)) - \nabla^\m V (\gamma(t),t)\hspace*{2cm}
\mathrm{(E^*)}
\]
must be complete (resp. forward complete; backward
complete).
\end{theorem}

When particularized to
autonomous systems, Theorem \ref{A02} extends the results by Weinstein and Marsden
in \cite{WM} and in \cite[Theorem 3.7.15]{AM}. Furthermore, in the
non--autonomous case, it generalizes widely the results by
Gordon in \cite{Go}.

\section{The Lorentzian Setting}\label{sec3}

From now till the end of this paper, let $(\m,g)$ be a Lorentzian manifold
and assume that
$F$ is a time--independent smooth $(1,1)$ tensor field on $\m$ and
$X$ is a time--independent smooth vector field on $\m$,
so that we consider the autonomous problem
\[
\hspace*{2.9cm}\frac{D\dot\gamma}{dt}(t)\ =\ F_{\gamma(t)}(\dot\gamma(t))
+ X_{\gamma(t)}.\hspace*{2.9cm}\mathrm{(\tilde{E})}
\]

First of all,  recall the following result which is a direct
consequence of the existence and uniqueness of solutions to second
order differential equations (see Lemma 4 and Remark 6 in
\cite{CRSfeb2012}, that apply also in the autonomous Lorentzian
case).

\begin{lemma}\label{vector_field}
There exists a unique vector field $G$ on the tangent bundle $T\m$
such that the curves $t\mapsto (\gamma(t),\dot\gamma(t))$ are the
integral curves of $G$ for any  solution $\gamma$ of equation
$(\tilde{E})$.
\end{lemma}

Recall that an integral curve $\rho$ of a vector field defined on some bounded
interval $[a,b)$, $b<+\infty$, can be extended to $b$ (as an
integral curve) if and only if there exists a sequence
$\{t_n\}_n$, $t_n\nearrow b$, such that $\{\rho(t_n)\}_n$
converges (see \cite[Lemma 1.56]{ON}). The following technical
result follows directly from this fact and Lemma \ref{vector_field}.

\begin{lemma}\label{extend}
Let $\gamma: [0,b) \to \m$ be a solution of equation
$(\tilde{E})$ with $0<b<+\infty$. The curve $\gamma$ can be
extended to $b$ as a solution of $(\tilde{E})$ if and only if
there exists a sequence $\{t_n\}_n \subset [0,b)$ such that $t_n
\to b^-$ and the sequence $\{\gamma(t_n),\dot\gamma(t_n)\}_n$ is
convergent in $T\m$.
\end{lemma}

Assume that the vector field $X$ derives from a smooth potential
$V: \m \to \R$, i.e.,  $X = - \nabla V$, and, hence, $(\tilde{E})$
reduces to $(E_0)$.
 Furthermore, suppose that
$F$ is skew--adjoint, so it results
$g(Y,F(Y)) = 0$ for any vector field $Y$.
In this setting, if $\gamma : (a,b) \to \R$ is a solution of
$(E_0)$, then
\[\begin{split}
\frac{d}{dt}(g(\dot\gamma(t),\dot\gamma(t) + 2 V(\gamma(t)))\ =& \
2 g(\dot\gamma(t),F(\dot\gamma(t)) - \nabla V(\gamma(t)))\\
&\quad + 2 g(\nabla V(\gamma(t)),\dot\gamma(t))\\
=& \ 0\qquad \hbox{for all $t \in (a,b)$}
\end{split}
\]
and a constant $c \in \R$ exists such that
\begin{equation}\label{constant}
g(\dot\gamma(t),\dot\gamma(t)) + 2 V(\gamma(t))\ = \ c \quad \hbox{for all $t \in (a,b)$.}
\end{equation}

Let us point out that, if $\m$ is a compact Lorentzian manifold,
the conservation law \eqref{constant} implies that
$g(\dot\gamma(t),\dot\gamma(t))$ is bounded. Anyway, unlike the
Riemannian case, this is not enough for applying Lemma
\ref{extend} and so proving the completeness of all the
inextendible solutions of $(E_0)$.

\begin{example}\label{incomplete}
(1) There are examples of compact Lorentzian manifolds which have
incomplete inextensible geodesics, i.e. solutions of $(E_0)$ in
the simplest case $F=0$, $V=0$. In fact, if $(\m,g)$ is a
Clifton--Pohl torus, then it is a compact Lorentzian manifold but
it is not geodesically complete (see \cite[Example 7.16]{ON}).

 (2)
There are examples of geodesically complete Lorentzian manifolds
$(\m,g)$ with bounded   $\|F\|^2=|\sum F^{\mu\nu}F_{\mu\nu}|$ and
$\|X\|^2=|\sum X^{\mu}X_{\mu}|$ such that they admit incomplete
inextensible solutions of $(\tilde{E})$. Indeed, it is enough to
consider $\m = \R^2$ and $g = dx \otimes dy + dy \otimes dx$ with
$F = 0$ and $X = 2 x^3 \frac{\partial}{\partial x}$. Direct
computations show that the corresponding problem $(\tilde{E})$ has
incomplete inextensible solutions.
\end{example}

It is a relevant fact that a compact Lorentzian manifold is
geodesically complete if it admits a timelike conformal vector
field (see \cite{RS}). Thus, it is natural to assume the existence of
such infinitesimal conformal symmetry to deal with the
extendibility of the solutions of $(E_0)$.

\begin{definition} \label{confKill}
A vector field $K$ is called {\em conformal Killing}, or simply
{\em conformal}, if the Lie derivative with respect to $K$, ${\cal
L}_K$, satisfies
\begin{equation}\label{Lie}
{\cal L}_K g= 2 \sigma g
\end{equation}
(the local flows of $K$ are conformal maps) for some $\sigma \in
C^\infty(\m)$.  In the case $\sigma = 0$, $K$ is called {\em
Killing}.
\end{definition}

In particular, if $K$ is a conformal vector field and $\gamma$ is
a geodesic, we have
\[
\frac{d}{dt}g(K,\dot\gamma)\ =\ \sigma(\gamma) g(\dot\gamma,\dot\gamma), \quad \hbox{with
$g(\dot\gamma,\dot\gamma)$ constant.}
\]
Hence, if $K$ is Killing then $g(K,\dot\gamma)$ is a constant.

More in general, if $\gamma : I \to \m$ is any curve, from \eqref{Lie}
it follows
\[
g(\nabla_{\dot\gamma}K,\dot\gamma)\ =\ \sigma(\gamma)\ g(\dot\gamma,\dot\gamma)
\]
which implies
\begin{equation}\label{Lie1}
\frac{d}{dt}g(K,\dot\gamma)\ =\ g(K,\frac{D\dot\gamma}{dt})\ +\ \sigma(\gamma)\ g(\dot\gamma,\dot\gamma).
\end{equation}

As already remarked, the compactness of $\m$ and the boundedness
of $g(\dot\gamma,\dot\gamma)$ are not enough to assure that the
image of $\dot\gamma$ is contained in a compact subset of the
tangent bundle $T\m$. So, in order to prove our main result, some
extra Lorentzian tools are needed.

\begin{lemma}\label{compact}
Let $(\m,g)$ be a (time-orientable) compact Lorentzian manifold
with a  timelike vector field $Z$ such that $g(Z,Z) = -1$. Assume
that $F$ is a skew--adjoint (1,1) tensor field and $V$ is a smooth
potential on $\m$. If $\gamma: I \to \m$ is a solution of $(E_0)$
such that $g(Z,\dot\gamma)$ is bounded in $I$, then there exists a
compact subset $C$ of $T\m$ which contains the image of
$\dot\gamma$.
\end{lemma}

\begin{proof}
As $\gamma$ is a solution of $(E_0)$ and $\m$ is compact, from
\eqref{constant} it follows that $g(\dot\gamma,\dot\gamma)$ is
bounded. On the other hand, consider the 1--form $\omega$
equivalent to $Z$ with respect to $g$, i.e. $\omega(X) = g(Z,X)$
for any vector field $X$, and the related
tensor field $g_R = g + 2 \omega \otimes \omega$ which is a
Riemannian metric on $\m$. By definition, it results
\[
g_R(\dot\gamma,\dot\gamma)\ =\ g(\dot\gamma,\dot\gamma) + 2
g(Z,\dot\gamma)^2;
\]
whence, in our hypotheses, $g_R(\dot\gamma,\dot\gamma)$ is bounded
on $I$. Thus, a constant $c> 0$ exists such that
\[
(\gamma(I),\dot\gamma(I)) \ \subset\ C, \quad C := \{(p,v) \in T\m :\; p\in\m,\ g_R(v,v) \le c\},
\]
with $C$ compact in $T\m$.
\end{proof}

\section{Proof of the Main Result}\label{sec4}

Now, we are ready to prove our main result.

\begin{proof}[{\sl Proof of Theorem \ref{main}}]
Without loss of generality, let $I= [0,b)$, $0<b<+\infty$, be the domain of a
for\-ward--inexten\-si\-ble solution $\gamma$ of $\mathrm{(E_0)}$.
As $F$ is skew--adjoint and null on $K$, it results
\[
g(K,F_\gamma(\dot\gamma)) \ =\ - g(F_\gamma(K),\dot\gamma) = 0,
\]
then from \eqref{Lie1} it follows
\[\begin{split}
\frac{d}{dt}g(K,\dot\gamma)\ =&\ g(K,F_\gamma(\dot\gamma)) - g(K,\nabla V(\gamma))\ +\ \sigma(\gamma)\ g(\dot\gamma,\dot\gamma)\\
=& \ - g(K,\nabla V(\gamma))\ +\ \sigma(\gamma)\ g(\dot\gamma,\dot\gamma).
\end{split}
\]
From the compactness of $\m$ and \eqref{constant} we have that both $g(K,\nabla V(\gamma))$ and
$\sigma(\gamma)\ g(\dot\gamma,\dot\gamma)$ are bounded on $I$, then $c_1 > 0$ exists
such that
\[
\big|\frac{d}{dt}g(K,\dot\gamma)\big|\ \le \ c_1\quad \hbox{on $I$;}
\]
whence, as $I$ is a bounded interval, a constant $c_2 > 0$ exists
such that
\begin{equation}\label{due}
|g(K,\dot\gamma)|\ \le \ c_2\quad \hbox{on $I$.}
\end{equation}
Now, let us define $Z = \frac{K}{\|K\|}$, with $\|K\|^2 = - g(K,K) > 0$ as $K$ is timelike.
So, $Z$ is a timelike vector field with $g(Z,Z) = -1$, and \eqref{due} implies
\[
|g(Z,\dot\gamma)|\ \le \ m c_2\quad \hbox{on $I$,}
\]
where $\displaystyle m = \max_\m \|K\|^{-1} > 0$ (which exists as $\m$ is compact).\\
Then, by applying Lemmas \ref{compact} and \ref{extend} we yield a contradiction.
\end{proof}

\begin{remark}
In particular, we may consider $F = 0$ and $V=0$ in Theorem
\ref{main} and, therefore,  this result extends the theorem on
completeness  in \cite{RS}.
\end{remark}



\begin{thebibliography}{0}


\bibitem{CRSfeb2012}  A.M. Candela, A. Romero and M. S\'anchez, Completeness
of the trajectories of particles coupled to a general force field, {\it Arch. Ration. Mech. Anal.}
(to appear). DOI:10.1007/s00205-012-0596-2.

\bibitem{CRSago2012}  A.M. Candela, A. Romero and M. S\'anchez,
Remarks on the completeness of plane waves and the trajectories of
accelerated particles in Riemannian manifolds, in: {\em Proc. Int.
Meeting on Differential Geometry} (University of C\'ordoba, C\'ordoba, 2012) pp. 27--38.

\bibitem{CS} A. Candela and M. S\'anchez, Geodesics in semi–-Riemannian
manifolds: geometric properties and variational tools, in: {\em
Recent Developments in pseudo–-Riemannian Geometry}, eds. D.V.
Alekseevsky \& H. Baum (Special Volume in the ESI– Series on
Mathematics and Physics, EMS Publ. House, Z\"urich, 2008) pp. 359--418.

\bibitem{SW} R.K. Sacks and H.H. Wu, {\it General Relativity for Mathematicians}
(Sprin\-ger--\-Ver\-lag, New York, 1977).

\bibitem{RS}
A. Romero and M. S\'anchez, Completeness of compact Lorentz
manifolds admitting a timelike conformal Killing vector field,
{\it Proc. Amer. Math. Soc.} {\bf 123} (1995), 2831--2833.

\bibitem{M}
J. Marsden, On completeness of homogeneous pseudo--Riemannian manifolds,
{\it Indiana Univ. J.} {\bf 22} (1972/73), 1065--1066.

\bibitem{C}
Y. Carri\`ere, Autour de la conjecture de L. Markus sur les vari\'et\'es affines,
{\it Invent. Math.} {\bf 95} (1989), 615--628.

\bibitem{Kl} B. Klingler, Compl\'etude des vari\'et\'es lorentziennes \`a courbure
              constante,
{\em Math. Ann.} {\bf 306} (1996), 353--370.

\bibitem{WM} A. Weinstein and J. Marsden, A comparison theorem for Hamiltonian vector fields,
\textit{Proc. Amer. Math. Soc.} \textbf{26} (1970), 629--631.

\bibitem{AM} R. Abraham and J. Marsden, \textit{Foundations of
Mechanics}, 2nd Ed. (Addison--Wesley Publishing
Co., Boston, 1987).

\bibitem{Go} W.B. Gordon, On the completeness of Hamiltonian
vector fields, \textit{Proc. Amer. Math. Soc.} \textbf{26}
(1970), 329--331.

\bibitem{ON} B. O'Neill, \textit{Semi--Riemannian Geometry with Applications to Relativity},
Pure Appl. Math. {\bf 103} (Academic Press Inc., New York, 1983).

\end{thebibliography}
\end{document}